\documentclass[12pt]{article}
\usepackage{amsmath, amssymb, amsthm, enumerate, booktabs, accents, framed, graphicx, array, color, multirow, mathrsfs, amsrefs, mathrsfs, cases, euscript, bm, multicol}

\usepackage[utf8]{inputenc}
\usepackage[colorlinks,citecolor=blue,urlcolor=blue]{hyperref}
\usepackage[usenames,dvipsnames,svgnames,table]{xcolor}
\usepackage{a4wide}

\allowdisplaybreaks[4]

\numberwithin{equation}{section} \theoremstyle{plain}
\newtheorem{theorem}{Theorem}[section]
\newtheorem{lemma}{Lemma}[section]
\newtheorem{corollary}{Corollary}[section]
\newtheorem{proposition}{Proposition}[section]

\newtheorem{remark}{Remark}[section]

\def\bC{\mathbb C}
\def\bE{\mathbb E}
\def\bN{\mathbb N}

\def\cC{\mathcal C}
\def\cI{\mathcal I}
\def\cJ{\mathcal J}

\def\y{\mathbf y}
\def\bi{\mathbf i}
\def\j{\mathbf j}
\def\Tr{\mathrm {Tr}}

\def\Var{\mathrm {Var}}
\def\deg{\mathrm {deg}}

\begin{document}
	
\title{On spectral distribution of sample covariance matrices from large dimensional and large $k$-fold tensor products}
\date{\today}
\author{Beno\^{\i}t Collins\thanks{Department of Mathematics, Kyoto University. E-mail: collins@math.kyoto-u.ac.jp} \and Jianfeng Yao\thanks{School of Data Science, The Chinese University of Hong Kong (Shenzhen). E-mail: jeffyao@cuhk.edu.cn} \and Wangjun Yuan\thanks{Department of Mathematics and Statistics, University of Ottawa. E-mail: ywangjun@connect.hku.hk}}
\maketitle

\begin{abstract}
    We study the eigenvalue distributions for sums of independent rank-one $k$-fold tensor products of large $n$-dimensional vectors.  Previous results in the literature assume that   $k=o(n)$ and show that the eigenvalue distributions converge to  the celebrated Mar\v{c}enko-Pastur law under appropriate moment conditions on the base vectors.  In this paper, motivated by 
    quantum information theory, we study the regime where
    $k$ grows faster, namely $k=O(n)$.
We show that the moment sequences of the eigenvalue distributions have a limit, which is different from the  Mar\v{c}enko-Pastur law. As a byproduct,  we show that the    Mar\v{c}enko-Pastur law limit holds if and only if $k=o(n)$ for this tensor model. The approach is based on the method of moments. 
\end{abstract}

\noindent{\bf AMS 2000 subject classifications:}\quad Primary 60B20; Secondary 15B52.

\medskip 

\noindent{\bf Keywords and phrases:}\quad   Large $k$-fold tensors; Eigenvalue distribution; Mar\v{c}enko-Pastur law; Quantum information theory. 

\section{Introduction}
\label{sec:intro}

For $n \in \bN$, let $\{\xi_1, \ldots, \xi_n\}$ be a family of i.i.d. centered complex random variables with unit variance. Denote $\y = \frac{1}{\sqrt{n}} (\xi_1, \ldots, \xi_n) \in \bC^n$. Let $\{\y_{\alpha}^{(l)}: 1 \le \alpha \le m, 1 \le l \le k\}$ be a family of i.i.d. copies of $\y$, and let $Y_{\alpha} = \y_{\alpha}^{(1)} \otimes \cdots \otimes \y_{\alpha}^{(k)}$ for $1 \le \alpha \le m$. Let $\{\tau_1, \tau_2, \ldots\}$ be a sequence of real numbers, and $Y = (Y_1, \ldots, Y_m)$ be a $n^k \times m$ matrix. Consider the $n^k \times n^k$ Hermitian matrix
\begin{align}\label{eq:matrix}
	M_{n,k,m} = \sum_{\alpha=1}^m \tau_{\alpha} Y_{\alpha} Y_{\alpha}^*.
\end{align}
Therefore, each $n^k$-dimensional vector $Y_\alpha$ is a $k$-fold tensor product of $n$-dimensional i.i.d. vectors, and $M_{n,k,m}$ is a sum of $m$ independent rank-$1$ Hermitian matrices of dimension $n^k$.

The simplest case of $k=1$ and $\tau_{\alpha} \equiv 1$ was studied in the seminal paper \cite{MP67} where the celebrated Mar\v{c}enko-Pastur law was derived under appropriate moment conditions on the entries of the base vector $\mathbf{y}$ and the limiting scheme where $n\to\infty$ and $m/n \to c>0$. Many subsequent improvements on the Mar\v{c}enko-Pastur law appeared in the literature, including \cite{Silv95}, \cite{BaiZhou08} and \cite{Pajor09}.
These papers were able to deal with the model \eqref{eq:matrix} with general sequence $\{\tau_{\alpha}: \alpha \in \bN_+\}$.
Notably, the latter paper extended the law to a broad family of $\mathbf{y}$-vectors, called {\em good vectors}, that includes the current setting with i.i.d. components.
Later, for the setting $\tau_{\alpha} \equiv 1$, the necessary and sufficient conditions on $Y_1$ that the Mar\v{c}enko-Pastur law serves as the limiting spectral distribution (LSD) of $M_{n,1,m}$ were carried out in \cite{Yaskov2016}. 

Recently, \cite{Lytova2018} considered the $k$-fold tensor model $M_{n,k,m}$ and established a LSD for its $n^k$ real-valued eigenvalues.
For the special case $\tau_{\alpha} \equiv 1$, the LSD is exactly the Mar\v{c}enko-Pastur law.
A central limit theorem (CLT) is also established for a class of linear spectral statistics following the approach of \cite{LytovaPastur09}.
The main setup \cite{Lytova2018} is that the tensor product number $k$ must be small enough compared to the space dimension $n$.
Precisely, $k/n \to 0$ is required for the validity of the LSD while $k\equiv 2$ is required when $n \to \infty$ for the validity of the CLT.
It is natural to consider the setting where $k$ is large.
Noting that large $k$ means more dependence among the entries of the vector $Y_1$, it is expected that the LSD of $M_{n,k,m}$ does not obey the Mar\v{c}enko-Pastur law when $k$ is large enough.
However, the martingale moment bound employed \cite{Lytova2018} is less useful when $k$ is large, and the method cannot be directly extended to the case $k = O(n)$.
The present paper deals with the case $k = O(n)$ for the model \eqref{eq:matrix} and shows that the empirical spectral distribution (ESD) of $M_{n,k,m}$ with $\tau_{\alpha} \equiv 1$ converges to the Mar\v{c}enko-Pastur law if and only if $k = o(n)$.
Therefore, the matrix $M_{n,k,m}$ with $k = O(n)$ can be seen as a new example of \emph{bad vectors}, for which the necessary and sufficient condition in \cite{Yaskov2016} does not hold.

Another motivation for studying the model \eqref{eq:matrix} is from quantum information theory.
In \cite{Hastings2012}, the model $M_{n,k,m}$ was introduced as a quantum analog of the classical probability problem of allocating $r$ balls into $s$ bins.
The random vector $Y_1$ was interpreted as random product states in $(\bC^n)^{\otimes k}$.
When $k$ is fixed and $m = cn^k$ for some $c>0$, \cite{Hastings2012} established the convergence in expectation of the normalized trace of moments $n^{-k} \Tr M_{n,k,m}^p$, which coincide with the corresponding moments of the Mar\v{c}enko-Pastur law.
In quantum physics and quantum information theory, it is natural to investigate the behavior of a large number of quantum states.
In \cite{Collins2011}, the quantum entanglement of structured random states was characterized, and the spectral density of the reduced state was studied when the number of the quantum states $k$ is large.
In \cite{Collins2013}, the asymptotic behavior of the average entropy of entanglement for elements of an ensemble of random states associated with a graph was studied when the dimension of the quantum subsystem is large.

This paper considers the scenario where $k$ grows to infinity quickly. Namely, we assume that there exist constants $c, d \in (0, \infty)$, such that
\begin{align} \label{eq-def-ratio}
	\dfrac{k}{n} \to d, \quad \dfrac{m}{n^k} \to c.
\end{align}
Let us add one more word of motivation for the choice of regime $k = O(n)$. In hindsight, this is a natural threshold in our quest for limiting theorems, and from the point of view of probability theory, our result extends existing results. 
However, the need for large $k$ is real in Quantum Information Theory. Actually, a typical real world scenario would be $n$ fixed (the dimension of the system and $k$ would go to infinity -- the system can be used many times, i.e., one works on regularized quantities.
In practice, we do not know how to fix $n$ and send $k$ to infinity, so we take this scenario as our next model.

Our approach is based on the method of moments. Under appropriate moment conditions on the base variable $\xi_1$ and the sequence of coefficients $\{\tau_\alpha\}$, we derive the limits for the spectral moments of $M_{n,k,m}$ under the limiting scheme~\eqref{eq-def-ratio}. A striking fact from our result is that contrary to all the previous results, \cite{Lytova2018}, the limiting spectral moments found here involve the $4$th moment of the base variable $\xi_1$.

\section{Almost sure convergence of spectral moments}
\label{sec:moment}

For fixed $\alpha = (\alpha_1, \ldots, \alpha_p)$,
for each value $t$ appearing in $\alpha$,
we count its frequency by
\begin{align}\label{eq:deg_t}
  \deg_t(\alpha) = \left|\{j \in [p]: \alpha_j = t \}\right|.
\end{align}

The main result of the paper is the following theorem.
\begin{theorem}\label{th:moment}

Assume $d>0$ and the following moment conditions hold.
\begin{enumerate}
\item For all $p \in \bN$. $p$-th moment $m_p=E [|\xi_1|^p] < \infty$.
\item
  For all $q \in \bN$,
  \begin{align} \label{eq-limit of tau}
	\dfrac{1}{m} \sum_{j=1}^m \tau_j^q \to m_q^{(\tau)}, \ m \to \infty.
  \end{align}
\end{enumerate}
We have
\begin{enumerate}
	\item[(i)]
	\begin{align} \label{eq-d>0-limit moment}
		\lim_{n \to \infty} \dfrac{1}{n^k} \bE \big[ \Tr M_{n,k,m}^p \big]
		=& \sum_{s=1}^p c^s \sum_{\alpha \in \cC_{s,p}^{(1)}} \left( \prod_{t=1}^s m_{\deg_t(\alpha)}^{(\tau)} \right) \exp \left( d(m_4-1) \sum_{t=1}^s \binom{\deg_t(\alpha)}{2} \right).
	\end{align}
	Here $\cC_{s,p}^{(1)}$ is a special class of graphs that will be defined in the course of the proof (see Lemma~\ref{lemma-Bai}).
	\item[(ii)]
	For all fixed $p \in \bN_+$, if $k \ge 2$, we have
	\begin{align*}
		\sum_{n=1}^{\infty} \Var \left( \dfrac{1}{n^k} \Tr M_{n,k,m}^p \right) < \infty.
	\end{align*}
\end{enumerate}
In particular, (i) and (ii) imply that $n^{-k} \Tr M_{n,k,m}^p$ converge almost surely to the limit given in the r.h.s. of \eqref{eq-d>0-limit moment}.
\end{theorem}

The proof of the theorem is given in Section~\ref{sec:proofs}.

\begin{remark}
Here, we require the random variable $\xi_1$ to have finite moments of any order. For most matrix models, such moment conditions can be removed by a standard truncation argument. However, the centralization step in the truncation argument fails for our matrix model. More precisely, let $\hat \y_{\alpha}^{(l)}$ be the truncated vector for all $1 \le \alpha \le m$ and $1 \le l \le k$, then we have the following identity
\begin{align} \label{eq-difference centralization}
	& \hat \y_{\alpha}^{(1)} \otimes \cdots \otimes \hat \y_{\alpha}^{(k)} - \big( \hat \y_{\alpha}^{(1)} - \bE \big[ \hat \y_{\alpha}^{(1)} \big] \big) \otimes \cdots \otimes \big( \hat \y_{\alpha}^{(k)} - \bE \big[ \hat \y_{\alpha}^{(k)} \big] \big) \nonumber \\
	=& \sum_{l=1}^k \hat \y_{\alpha}^{(1)} \otimes \cdots \otimes \hat \y_{\alpha}^{(l-1)} \otimes \bE \big[ \hat \y_{\alpha}^{(l)} \big] \otimes \big( \hat \y_{\alpha}^{(l+1)} - \bE \big[ \hat \y_{\alpha}^{(l+1)} \big] \big) \otimes \cdots \otimes \big( \hat \y_{\alpha}^{(k)} - \bE \big[ \hat \y_{\alpha}^{(k)} \big] \big).
\end{align}
Note that for $1 \le l \le k$, the $n^k \times m$ matrix, whose $\alpha$-th column is
\begin{align*}
	\hat \y_{\alpha}^{(1)} \otimes \cdots \otimes \hat \y_{\alpha}^{(l-1)} \otimes \bE \big[ \hat \y_{\alpha}^{(l)} \big] \otimes \big( \hat \y_{\alpha}^{(l+1)} - \bE \big[ \hat \y_{\alpha}^{(l+1)} \big] \big) \otimes \cdots \otimes \big( \hat \y_{\alpha}^{(k)} - \bE \big[ \hat \y_{\alpha}^{(k)} \big] \big),
\end{align*}
has rank at most $n^{k-1}$. Thus, the $n^k \times m$ matrix, whose $\alpha$-th column is \eqref{eq-difference centralization}, has rank at most $kn^{k-1}$. Hence, by \cite[Theorem A.44]{Bai2010}, the sup norm of the difference of the cumulative distribution functions in the centralization step does not exceed $k/n$, which is not negligible when $d > 0$.
\end{remark}

Let $\theta= e^{d(m_4-1)}$. The limit moments in~\eqref{eq-d>0-limit moment}, say $\gamma_p$, are  polynomial functions of  $\theta$.
The first four moments are 
\begin{align*}
	\gamma_1 & = 1,\\
	\gamma_2 & = c\theta+c^2, \\
	\gamma_3 & = c\theta^3+3c^2\theta+c^3,\\
	\gamma_4 &=  c\theta^6 + 4c^2 \theta^3 + 2c^2\theta^2 + 6c^3\theta +c^4. 
\end{align*}
However, for higher exponent $p$, some computing code is needed to find an explicit expression for $\gamma_p$.

\begin{remark}
The limiting moment sequence $(\gamma_p)$ grow to infinity extremely fast with $p$. To see this, let $\tau_\alpha \equiv 1$. Then $\gamma_p$ is lower bounded by the first term  ($s=1$) in \eqref{eq-d>0-limit moment} is $c\theta^{p(p-1)/2}$. Thus the Carleman's condition, that is, $\sum_{p=1}^\infty \gamma_{2p}^{-1/(2p)}=\infty$, is not satisfied (see also \cite{Lin2017}).
In particular, it is not clear whether the moment sequence   $(\gamma_p)$ uniquely determines a limiting distribution.   However, by the convergence of the moments, we know that the sequence of eigenvalue distributions is tight (almost surely).
\end{remark}

As a  byproduct of our moment method, we give an alternative method for deriving a limiting spectral distribution in the case of $d=0$, a result already given in \cite{Lytova2018} using the Stieltjes transform method.

\begin{proposition}\label{prop:d=0}
Assume $d=0$. We have
\begin{enumerate}
	\item[(i)]
	\begin{align} \label{eq-d=0-limit moment}
		\lim_{n \to \infty} \dfrac{1}{n^k} \bE \big[ \Tr M_{n,k,m}^p \big]
		= \sum_{s=1}^p c^s \sum_{\alpha \in \cC_{s,p}^{(1)}} \left( \prod_{t=1}^s m_{\deg_t(\alpha)}^{(\tau)} \right).
	\end{align}
	Here $\cC_{s,p}^{(1)}$ is a special class of graphs defined later in Lemma~\ref{lemma-Bai}.
	\item[(ii)]
	For all fixed $p \in \bN_+$, if $k \ge 2$, we have
	\begin{align*}
		\sum_{n=1}^{\infty} \Var \left( \dfrac{1}{n^k} \Tr M_{n,k,m}^p \right) < \infty.
	\end{align*}
\end{enumerate}
In particular, (i) and (ii) imply that $n^{-k} \Tr M_{n,k,m}^p$ converge almost surely to the limit given in the r.h.s. of \eqref{eq-d=0-limit moment}.
\end{proposition}

The proof of the proposition is given in Appendix~\ref{sec:d=0}.

Finally, comparing the two cases $d>0$ and $d=0$, it is worth noticing that the fourth moment $m_4$ contributes to the limiting spectral moments only in the case of  $d>0$  (Theorem~\ref{th:moment}).

\section{Proof of Theorem~\ref{th:moment}}
\label{sec:proofs}

Before proceeding to the proof, we introduce some preliminaries on graph combinatorics.
We first introduce some terminologies and notations from graph theory. Denote $[m]$ by the set of integers in the closed interval $[1,m]$. We call $\alpha=(\alpha_1, \ldots, \alpha_p) \in [m]^p$ a sequence of length $p$ with vertices $\alpha_j$ for $1 \le j \le p$. We denote by $|\alpha|$ the number of distinct elements in $\alpha$. If $s = |\alpha|$, then we call $\alpha$ an $s$-sequence. Two sequence are equivalent if one becomes the other by a suitable permutation on $[m]$. The sequence $\alpha$ is canonical if $\alpha_1=1$ and $\alpha_u \le \max\{\alpha_1, \ldots, \alpha_{u-1}\} + 1$ for $u \ge 2$. We denote by $\cC_{s,p}$ the set of all canonical $s$-sequences of length $p$, and denote $\cJ_{s,p}(m)$ by the set of all $s$-sequences $\alpha \in [m]^p$. Then
\begin{align} \label{eq-0.3}
	[m]^p = \bigcup_{s=1}^p \cJ_{s,p}(m).
\end{align}
From the definition above, one can see that the set of distinct vertices of a canonical $s$-sequence is $[s]$. Denote by $\cI_{s,m}$ the set of injective maps from $[s]$ to $[m]$. For a canonical $s$-sequence $\alpha$ and a map $\varphi \in \cI_{s,m}$, we call $\varphi(\alpha)$ the $s$-sequence $(\varphi(\alpha_1), \ldots, \varphi(\alpha_p))$. For each canonical $s$-sequence, its image under the maps in $\cI_{s,m}$ gives all its equivalent sequence, and hence its equivalent class of sequence in $[m]^p$ has exactly $m(m-1) \cdots (m-s+1)$ distinct elements.

Let $\alpha = (\alpha_1, \ldots, \alpha_p) \in [m]^p$ be a fixed canonical $s$-sequence. For $i_1 = (i_1^{(1)}, \ldots, i_1^{(p)}) \in [n]^p$, draw two parallel lines, referring to the $\alpha$-line and the $i$-line. Plot $i_1^{(1)}, \ldots, i_1^{(p)}$ on the $i$-line and $\alpha_1, \ldots, \alpha_p$ on the $\alpha$-line. Draw $p$ down edges from $\alpha_u$ to $i_1^{(u)}$ and $p$ up edges from $i_1^{(u)}$ to $\alpha_{u+1}$ for $1 \le u \le p$. Here, we use the convention that $\alpha_1 = \alpha_{p+1}$. We call a down edge from $\alpha_u$ to $i_1^{(u)}$ a down innovation if $i_1^{(u)}$ is different from $i_1^{(1)}, \ldots, i_1^{(u-1)}$. We also call an up edge from $i_1^{(u)}$ to $\alpha_{u+1}$ an up innovation if $\alpha_{u+1}$ is different from $\alpha_1, \ldots, \alpha_u$. We denote the graph by $g(i_1,\alpha)$ and call such graph a $\Delta(p;\alpha)$-graph. Two graphs $g(i_1,\alpha)$ and $g(i_1',\alpha)$ are said to be equivalent if the two sequences $i_1$ and $i_1'$ are equivalent. We write $g(i_1,\alpha) \sim g(i_1',\alpha)$ if they are equivalent. For each equivalent class, we choose the canonical graph such that $i_1 = (i_1^{(1)}, \ldots, i_1^{(p)}) \in [n]^p$ is a canonical $r$-sequence for some $r \in \bN_+$. A canonical $\Delta(p;\alpha)$-graph is denoted by $\Delta(p,r,s;\alpha)$ if it has $r$ noncoincident $i$-vertices and $s$ noncoincident $\alpha$-vertices. We classify $\Delta(p,r,s;\alpha)$-graphs into the following five categories.
\begin{enumerate}
	\item[] $\Delta_1(p,s;\alpha)$. $\Delta(p;\alpha)$-graphs in which each down edge must coincide with one and only one up edge. If we glue the coincident edges, the resulting graph is a tree of $p$ edges and $p+1$ vertices, which implies $r+s=p+1$.
	
	\item[] $\Delta_2(p,r,s;\alpha)$. $\Delta(p;\alpha)$-graphs that contain at least one single edge.
	
	\item[] $\Delta_3(p,s;\alpha)$. $\Delta(p;\alpha)$-graphs such that the number of the edges between two vertices is $0$ or $2$. If we glue the coincident edges, the resulting graph is a connected graph with exactly one cycle. The graph has $p$ edges and $p$ vertices, which implies $r+s=p$.
	
	\item[] $\Delta_4(p,s;\alpha)$. $\Delta(p;\alpha)$-graphs that have two up edges and two down edges with the same endpoints and all other down edges coincide with one and only one up edge. If we glue the coincident edges, the resulting graph is a tree of $p-1$ edges and $p$ vertices, implying $r+s=p$.
	
	\item[] $\Delta_5(p,r,s;\alpha)$. $\Delta(p;\alpha)$-graphs that do not belong to the categories above. In this case, the graph has at most $p-1$ vertices, since the in-degree equals to out-degree and at least $2$ for all vertices. Thus, $r+s \le p-1$.
\end{enumerate}

Next, we determine the number of sequence $i_1 \in \cC_{p+1-s,p}$  such that $g(i_1,\alpha) \in \Delta_1(p,s;\alpha)$ for given sequence $\alpha \in \cC_{s,p}$. We have the following lemma, which is motivated by \cite[Lemma 3.4]{Bai2010}.

\begin{lemma} \label{lemma-Bai}
	For any sequence $\alpha \in \cC_{s,p}$, there is at most one sequence $i_1 \in \cC_{p+1-s,p}$ such that $g(i_1,\alpha) \in \Delta_1(p,s;\alpha)$. We denote by $\cC_{s,p}^{(1)}$ the set of such canonical sequences $\alpha$. Then the number of the elements in $\cC_{s,p}^{(1)}$ is
	\begin{align*}
		\dfrac{1}{p} \binom{p}{s-1} \binom{p}{s}.
	\end{align*}
\end{lemma}

\begin{proof}
	We denote
	\begin{align*}
		\Delta_1(p,s) = \bigcup_{\alpha \in \cC_{s,p}} \Delta_1(p,s;\alpha).
	\end{align*}
	We define a pair of characteristic sequences $\{u_1, \ldots, u_p\}$ and $\{d_1, \ldots, d_p\}$ by
	\begin{align} \label{eq-def-u sequence}
		u_l =
		\begin{cases}
			1, & \alpha_{l+1} = \max \{ \alpha_1, \ldots, \alpha_l \} +1, \\
			0, & \mathrm{otherwise},
		\end{cases}
	\end{align}
	and
	\begin{align} \label{eq-def-d sequence}
		d_l =
		\begin{cases}
			-1, & \alpha_l \notin \{1, \alpha_{l+1}, \ldots, \alpha_p\}, \\
			0, & \mathrm{otherwise}.
		\end{cases}
	\end{align}
	By definition, we always have $u_p=0$, and since $\alpha_1 = 1$, we always have $d_1=0$.
	
	For a graph belongs to $\Delta_1(p,s)$, there are exactly $s-1$ up innovations and hence there are $s-1$ $u$-variables equal to $1$ and $s-1$ $d$-variables equal to $-1$. From its definition, one sees that $d_l=-1$ means that after plotting the $l$-th down edge $(\alpha_l, i_1^{(l)})$, the future path will never revisit the vertices $\alpha_l$, which means that the edge $(\alpha_l, i_1^{(l)})$ must coincide with the up innovation leading to the vertex $\alpha_l$. Since there are $r = p+1-s$ down innovations to lead out the $r$ $i_1$-vertices, $d_l = 0$ implies that the edge $(\alpha_l, i_1^{(l)})$ must be a down innovation. Therefore, $d_l=-1$ must follow a $u_j=1$ for some $j<l$, which leads to the restriction of the pair of characteristic sequences
	\begin{align} \label{eq-restriction-characteristic sequence}
		u_1 + \cdots + u_{l-1} + d_2 + \cdots + d_l \ge 0, \qquad 2 \le l \le p.
	\end{align}
	
	Next, we show that each pair of characteristic sequences satisfying \eqref{eq-restriction-characteristic sequence} defines a graph in $\Delta_1(p,s)$ uniquely.
	
	Firstly, we have $i_1^{(1)} = 1$. Besides, $\alpha_2 = 1$ if $u_1=0$ and $\alpha_2 = 2$ if $u_1=1$. We use induction to determine the unique graph in $\Delta_1(p,s)$. Suppose that the first $l$ pairs of the down and up edges are uniquely determined by the two sequences $\{u_1, \ldots u_l\}$ and $\{d_1, \ldots d_l\}$, and the subgraph of the first $l$ pairs of down and up edges satisfies the following properties
	\begin{enumerate}
		\item[(a1)] The subgraph is connected, and the unidirectional noncoincident edges form a tree.
		\item[(a2)] If $\alpha_{l+1} = 1$, then each down edge coincides with an up edge, which means that the subgraph not have single	innovations.
		\item[(a3)] If $\alpha_{l+1} \neq 1$, then from the vertex $\alpha_1$ to $\alpha_{l+1}$, there is only one path (chain without cycles) of down-up-down-up single innovations and all other down edges coincide with an up edge.
	\end{enumerate}
	We consider the following four cases to determine the $(l+1)$-th pair of down and up edges.
	\begin{enumerate}
		\item[Case 1.] $d_{l+1}=0$ and $u_{l+1}=1$. Then both edges of the $(l+1)$-th pair are	innovations, which implies that $i_1^{(l+1)} = |\{j \le l+1: d_j=0\}|$ and $\alpha_{l+2} = 1 + |\{j \le l+1: u_j=1\}|$. After adding the two edges, the subgraph with the first $l+1$ pairs of down and up edges satisfies the properties (a1) - (a3).
		
		\item[Case 2.] $d_{l+1}=0$ and $u_{l+1}=0$. Then the down edge $(\alpha_{l+1}, i_1^{(l+1)})$ is an innovation so $i_1^{(l+1)} = |\{j \le l+1: d_j=0\}|$. The up edge $(i_1^{(l+1)}, \alpha_{l+2})$ is not an innovation so $\alpha_{l+2}$ coincides with some vertex $\alpha_j$ for $1 \le j \le l+1$. If $\alpha_j \neq \alpha_{l+1}$, then there is a path $i_1^{(l)} \to \alpha_{l+1} \to i_1^{(l+1)} \to \alpha_j$. Also note that there should be a path connecting $\alpha_j$ and $i_1^{(l)}$ in the subgraph of first $l$ pairs of down and up edges. The two paths with undirectional noncoincident edges will lead to a circle, which is a contradiction. Hence, $\alpha_{l+2} = \alpha_j = \alpha_{l+1}$ and the $(l+1)$-th up edge coincide with the $(l+1)$-th down edge. After adding the two edges, the subgraph with the first $l+1$ pairs of down and up edges satisfies the properties (a1) - (a3).
		
		\item[Case 3.] $d_{l+1}=-1$ and $u_{l+1}=1$. By \eqref{eq-restriction-characteristic sequence}, we have
		\begin{align*}
			u_1 + \cdots + u_l + d_2 + \cdots + d_l \ge 1.
		\end{align*}
		Hence, there is at least one vertex in $\{\alpha_2, \ldots, \alpha_{l+1}\}$ which is not $\alpha_1$, such that the vertex will be visited in the last $p-(l+1)$ pairs of down and up edges. Thus, by property (a1) and (a2), we have $\alpha_{l+1} \neq \alpha_1$. Hence, there must be a single up innovation leading to the vertex $\alpha_{l+1}$. We denote the up innovation by $(i_1^{(j)}, \alpha_{l+1})$ for some $1 \le j \le l$. Hence, we choose $i_1^{(l+1)} = i_1^{(j)}$, which means that the down edge starting from $\alpha_{l+1}$ coincides with the up innovation leading to $\alpha_{l+1}$. Besides, the $(l+1)$-th up edge is an innovation, which starts from $i_1^{(l+1)}$ and ends in $\alpha_{l+2} = 1 + |\{j \le l+1: u_j=1\}|$. After adding the two edges, the subgraph with the first $l+1$ pairs of down and up edges satisfies the properties (a1) - (a3).
		
		\item[Case 4.] $d_{l+1}=-1$ and $u_{l+1}=0$. As discussed in Case 3, the $(l+1)$-th down edge coincides with the only up innovation ended at $\alpha_{l+1}$. Before this up innovation, there must be a single down innovation by property (a3). Then the up edge can be drawn to coincide with this down innovation. If the path of single innovations of the subgraph with first $l$ pairs of down and up edges has only one pair of down-up innovations, then $\alpha_{l+2} = 1$, and hence the subgraph with first $l+1$ pairs of down and up edges has no single innovations, which implies that the properties (a1) and (a2) hold. Otherwise, $\alpha_{l+2} \neq 1$ and the subgraph with first $l+1$ pairs of down and up edges satisfies properties (a1) and (a3).
	\end{enumerate}
	
	By induction, it is shown that two characteristic sequences $\{u_1, \ldots, u_p\}$ and $\{d_1, \ldots, d_p\}$ satisfying \eqref{eq-restriction-characteristic sequence} determine a graph in $\Delta_1(p,s)$ uniquely.
	
	For a given sequence $\alpha = (\alpha_1, \ldots, \alpha_p) \in \cC_{s,p}$, the two characteristic sequences are uniquely determined by \eqref{eq-def-u sequence} and \eqref{eq-def-d sequence}. This shows that there is at most one sequence $i_1 \in \cC_{p+1-s,p}$ such that $g(i_1,\alpha) \in \Delta_1(p,s;\alpha)$. More precisely, there is exactly one sequence $i_1 \in \cC_{p+1-s,p}$ such that $g(i_1,\alpha) \in \Delta_1(p,s;\alpha)$ if and only if the pair of two characteristic sequences satisfies \eqref{eq-restriction-characteristic sequence}.
	
	Moreover, by \cite[Lemma 3.5]{Bai2010}, the number of the elements in $\cC_{s,p}^{(1)}$ equals to the number of graphs in
	\begin{align*}
		\bigcup_{\alpha \in \cC_{s,p}} \Delta_1(p,s;\alpha),
	\end{align*}
	which is
	\begin{align*}
		\dfrac{1}{s} \binom{p}{s-1} \binom{p-1}{s-1}
		= \dfrac{1}{p} \binom{p}{s-1} \binom{p}{s}.
	\end{align*}
\end{proof}

\begin{remark}
There exists $\alpha \in \cC_{s,p} \setminus \cC_{s,p}^{(1)}$, such that the $u$ sequence and $d$ sequence defined by \eqref{eq-def-u sequence} and \eqref{eq-def-d sequence} respectively satisfy the restriction \eqref{eq-restriction-characteristic sequence}. Consider the canonical $2$-sequence $\alpha = (1,2,1,2)$, the corresponding $u$ sequence is $\{1,0,0,0\}$ and the corresponding $d$ sequence is $\{0,0,0,-1\}$, which satisfy \eqref{eq-restriction-characteristic sequence}. However, the graph of $\alpha$ belongs to $\Delta_3(4,2;\alpha)$, which means that $\alpha \in \cC_{2,4} \setminus \cC_{2,4}^{(1)}$. Indeed, the $\alpha$ sequence of the graph in $\Delta_1(4,2;\alpha)$ corresponding to the pair of characteristic sequence is $(1,2,2,2)$.
\end{remark}

\subsection{Proof of (i)}
We will use the index $\alpha, \beta$ for the columns of $Y$, which should be in $[m]$. We will also use the index $\bi$, a multiple index of the form $\bi = (i_1, \ldots, i_k)$, for the rows of $Y$. Then $\bi$ should be in $[n^k]$.
For any $p \in \bN$, we compute the moment
\begin{align*}
	\dfrac{1}{n^k} \bE \big[ \Tr M_{n,k,m}^p \big].
\end{align*}
We use the convention that $\bi^{(p+1)} = \bi^{(1)}$. We have
\begin{align} \label{eq-moment-0.2}
	\dfrac{1}{n^k} \bE \big[ \Tr M_{n,k,m}^p \big]
	=& \dfrac{1}{n^k} \sum_{\alpha_1, \ldots, \alpha_p = 1}^m \left( \prod_{t=1}^p \tau_{\alpha_t} \right) \bE \Big[ \Tr \big( Y_{\alpha_1} Y_{\alpha_1}^* \cdots Y_{\alpha_p} Y_{\alpha_p}^* \big) \Big] \nonumber \\
	=& \dfrac{1}{n^k} \sum_{\alpha_1, \ldots, \alpha_p = 1}^m \left( \prod_{t=1}^p \tau_{\alpha_t} \right) \bE \Bigg[ \sum_{\bi^{(1)}, \ldots, \bi^{(p)} = 1}^{n^k} \prod_{t=1}^p Y_{\bi^{(t)} \alpha_t} \overline{Y_{\bi^{(t+1)} \alpha_t}} \Bigg] \nonumber \\
	=& \dfrac{1}{n^k} \sum_{\alpha_1, \ldots, \alpha_p = 1}^m \left( \prod_{t=1}^p \tau_{\alpha_t} \right) \bE \Bigg[ \sum_{\bi^{(1)}, \ldots, \bi^{(p)} = 1}^{n^k} \prod_{t=1}^p \prod_{l=1}^k \bigg( \big( \y_{\alpha_t}^{(l)} \big)_{i^{(t)}_l} \overline{\big( \y_{\alpha_t}^{(l)} \big)_{i^{(t+1)}_l}} \bigg) \Bigg] \nonumber \\
	=& \dfrac{1}{n^k} \sum_{\alpha_1, \ldots, \alpha_p = 1}^m \left( \prod_{t=1}^p \tau_{\alpha_t} \right) \bE \Bigg[ \prod_{l=1}^k \sum_{i_l^{(1)}, \ldots, i_l^{(p)} = 1}^n \prod_{t=1}^p \bigg( \big( \y_{\alpha_t}^{(l)} \big)_{i^{(t)}_l} \overline{\big( \y_{\alpha_t}^{(l)} \big)_{i^{(t+1)}_l}} \bigg) \Bigg] \nonumber \\
	=& \dfrac{1}{n^k} \sum_{\alpha_1, \ldots, \alpha_p = 1}^m \left( \prod_{t=1}^p \tau_{\alpha_t} \right) \left( \bE \left[ \sum_{i_1^{(1)}, \ldots, i_1^{(p)} = 1}^n \prod_{t=1}^p \bigg( \big( \y_{\alpha_t}^{(1)} \big)_{i^{(t)}_1} \overline{\big( \y_{\alpha_t}^{(1)} \big)_{i^{(t+1)}_1}} \bigg) \right] \right)^k.
\end{align}
Here, we use the i.i.d. setting in the last equality.

For two sequences $\alpha = (\alpha_1, \ldots, \alpha_p) \in [m]^p$ and $i_1 = (i_1^{(1)}, \ldots, i_1^{(p)}) \in [n]^p$, we denote
\begin{align} \label{eq-def-E(i_1,alpha)}
	E(i_1,\alpha) = \bE \left[ \prod_{t=1}^p \bigg( \big( \y_{\alpha_t}^{(1)} \big)_{i^{(t)}_1} \overline{\big( \y_{\alpha_t}^{(1)} \big)_{i^{(t+1)}_1}} \bigg) \right].
\end{align}
An observation is that $E(i_1,\alpha)=E(i_1',\alpha')$ if the two sequences $i_1$ and $\alpha$ are equivalent to $i_1'$ and $\alpha'$ respectively. Hence, by \eqref{eq-moment-0.2} and \eqref{eq-0.3}, we have
\begin{align} \label{eq-moment-0.4}
	\dfrac{1}{n^k} \bE \big[ \Tr M_{n,k,m}^p \big]
	=& \dfrac{1}{n^k} \sum_{s=1}^p \sum_{\alpha \in \cJ_{s,p}(m)} \left( \prod_{t=1}^p \tau_{\alpha_t} \right) \left( \sum_{r=1}^p \sum_{i_1 \in \cJ_{r,p}(n)} E(i_1,\alpha) \right)^k \nonumber \\
	=& \dfrac{1}{n^k} \sum_{s=1}^p \sum_{\alpha \in \cC_{s,p}} \left( \sum_{\varphi \in \cI_{s,m}} \prod_{t=1}^p \tau_{\varphi(\alpha_t)} \right) \left( \sum_{r=1}^p n \cdots (n-r+1) \sum_{i_1 \in \cC_{r,p}} E(i_1,\alpha) \right)^k.
\end{align}

We first compute the sum on $i_1$ in \eqref{eq-moment-0.4}. For any sequences $\alpha \in \cC_{s,p}$, we can split the sum according to the category of the graph $g(i_1,\alpha)$ in the following way:
\begin{align} \label{eq-split of sum}
    \sum_{i_1 \in \cC_{r,p}} E(i_1,\alpha)
    =& \sum_{i_1 \in \cC_{r,p} \atop g(i_1,\alpha) \in \Delta_1(p,s;\alpha)} E(i_1,\alpha)
    + \sum_{i_1 \in \cC_{r,p} \atop g(i_1,\alpha) \in \Delta_3(p,s;\alpha)} E(i_1,\alpha) \nonumber \\
    &+ \sum_{i_1 \in \cC_{r,p} \atop g(i_1,\alpha) \in \Delta_4(p,s;\alpha)} E(i_1,\alpha)
    + \sum_{i_1 \in \cC_{r,p} \atop g(i_1,\alpha) \in \Delta_5(p,r,s;\alpha)} E(i_1,\alpha)
\end{align}
We deal with the four terms on the right hand side of \eqref{eq-split of sum} one by one. For $i_1 \in \cC_{r,p}$ with $g(i_1,\alpha) \in \Delta_1(p,s;\alpha)$, by the definition of $\Delta_1(p,s;\alpha)$, we have $E(i_1,\alpha) = n^{-p}$ and $p = r+s-1$. Hence,
\begin{align} \label{eq-sum Delta_1}
    \sum_{i_1 \in \cC_{r,p} \atop g(i_1,\alpha) \in \Delta_1(p,s;\alpha)} E(i_1,\alpha)
    = n^{1-r-s} \sum_{i_1 \in \cC_{r,p} \atop g(i_1,\alpha) \in \Delta_1(p,s;\alpha)} 1.
\end{align}
For $i_1 \in \cC_{r,p}$ with $g(i_1,\alpha) \in \Delta_3(p,s;\alpha)$, by the definition of $\Delta_3(p,s;\alpha)$, we have $E(i_1,\alpha) = n^{-p}$ and $p = r+s$. Hence,
\begin{align} \label{eq-sum Delta_3}
    \sum_{i_1 \in \cC_{r,p} \atop g(i_1,\alpha) \in \Delta_3(p,s;\alpha)} E(i_1,\alpha)
    = n^{-r-s} \sum_{i_1 \in \cC_{r,p} \atop g(i_1,\alpha) \in \Delta_3(p,s;\alpha)} 1.
\end{align}
For $i_1 \in \cC_{r,p}$ with $g(i_1,\alpha) \in \Delta_4(p,s;\alpha)$, by the definition of $\Delta_4(p,s;\alpha)$, we have $E(i_1,\alpha) = n^{-p} m_4$ and $p = r+s$. Hence,
\begin{align} \label{eq-sum Delta_4}
    \sum_{i_1 \in \cC_{r,p} \atop g(i_1,\alpha) \in \Delta_4(p,s;\alpha)} E(i_1,\alpha)
    = n^{-r-s} m_4 \sum_{i_1 \in \cC_{r,p} \atop g(i_1,\alpha) \in \Delta_4(p,s;\alpha)} 1.
\end{align}
For $i_1 \in \cC_{r,p}$ with $g(i_1,\alpha) \in \Delta_5(p,r,s;\alpha)$, by the definition of $\Delta_5(p,r,s;\alpha)$, we have $E(i_1,\alpha) = O(n^{-p})$ and $p \ge r+s+1$. Hence,
\begin{align} \label{eq-sum Delta_5}
    \sum_{i_1 \in \cC_{r,p} \atop g(i_1,\alpha) \in \Delta_5(p,r,s;\alpha)} E(i_1,\alpha)
    = O(n^{-p}) \sum_{i_1 \in \cC_{r,p} \atop g(i_1,\alpha) \in \Delta_5(p,r,s;\alpha)} 1.
\end{align}
Substituting \eqref{eq-sum Delta_1}, \eqref{eq-sum Delta_3}, \eqref{eq-sum Delta_4} and \eqref{eq-sum Delta_5} to \eqref{eq-split of sum}, and noting that
\begin{align*}
    n \cdots (n-r+1) =  n^r \left( 1 - \dfrac{r(r-1)}{2n} + O(n^{-2}) \right),
\end{align*}
we have
\begin{align}
	& \sum_{r=1}^p n \cdots (n-r+1) \sum_{i_1 \in \cC_{r,p}} E(i_1,\alpha) \nonumber \\
	=& n^{1-s} \left( 1 - \dfrac{(p+1-s)(p-s)}{2n} + O(n^{-2}) \right) \sum_{i_1 \in \cC_{p+1-s,p} \atop g(i_1,\alpha) \in \Delta_1(p,s;\alpha)} 1 \nonumber \\
	&+ n^{-s} \left( 1 - \dfrac{(p-s)(p-s-1)}{2n} + O(n^{-2}) \right) \sum_{i_1 \in \cC_{r,p} \atop g(i_1,\alpha) \in \Delta_3(p,s;\alpha)} 1 \nonumber \\
	&+ n^{-s} \left( 1 - \dfrac{(p-s)(p-s-1)}{2n} + O(n^{-2}) \right) m_4 \sum_{i_1 \in \cC_{r,p} \atop g(i_1,\alpha) \in \Delta_4(p,s;\alpha)} 1 \nonumber \\
	&+ \sum_{r=1}^{p-s-1} n^r \left( 1 - \dfrac{r(r-1)}{2n} + O(n^{-2}) \right) \sum_{i_1 \in \cC_{r,p} \atop g(i_1,\alpha) \in \Delta_5(p,r,s;\alpha)} O(n^{-p}) \nonumber \\
	=& n^{1-s} \left( 1 - \dfrac{(p+1-s)(p-s)}{2n} \right) \sum_{i_1 \in \cC_{p+1-s,p} \atop g(i_1,\alpha) \in \Delta_1(p,s;\alpha)} 1 \nonumber \\
	& + n^{-s} \left( \sum_{i_1 \in \cC_{r,p} \atop g(i_1,\alpha) \in \Delta_3(p,s;\alpha)} 1 + m_4 \sum_{i_1 \in \cC_{r,p} \atop g(i_1,\alpha) \in \Delta_4(p,s;\alpha)} 1 \right)
	+ O(n^{-s-1}) \label{eq-d>0-sum i_1} \\
	=&
	\begin{cases}
		O(n^{1-s}), & \alpha \in \cC_{s,p}^{(1)}, \\
		O(n^{-s}), & \alpha \in \cC_{s,p} \setminus \cC_{s,p}^{(1)},
	\end{cases} \label{eq-d>0-sum i_1-order}
\end{align}
where the last equality follows from Lemma \ref{lemma-Bai}.
Hence, by \eqref{eq-d>0-sum i_1}, \eqref{eq-d>0-sum i_1-order} and Lemma \ref{lemma-Bai}, \eqref{eq-moment-0.4} can be written as
\begin{align} \label{eq-d>0-moment-2.2}
	\dfrac{1}{n^k} \bE \big[ \Tr M_{n,k,m}^p \big]
	=& \dfrac{1}{n^k} \sum_{s=1}^p \sum_{\alpha \in \cC_{s,p}} \left( \sum_{\varphi \in \cI_{s,m}} \prod_{t=1}^p \tau_{\varphi(\alpha_t)} \right) \left( \sum_{r=1}^p n \cdots (n-r+1) \sum_{i_1 \in \cC_{r,p}} E(i_1,\alpha) \right)^k \nonumber \\
	=& \sum_{s=1}^p \left( \dfrac{m}{n^k} \right)^s (1+o(1)) \sum_{\alpha \in \cC_{s,p}^{(1)}} \left( \dfrac{1}{m^s} \sum_{\varphi \in \cI_{s,m}} \prod_{t=1}^p \tau_{\varphi(\alpha_t)} \right) \nonumber \\
	& \left( 1 - \dfrac{(p+1-s)(p-s)}{2n} + \dfrac{1}{n} \left( \sum_{i_1 \in \cC_{r,p} \atop g(i_1,\alpha) \in \Delta_3(p,s;\alpha)} 1 + m_4 \sum_{i_1 \in \cC_{r,p} \atop g(i_1,\alpha) \in \Delta_4(p,s;\alpha)} 1 \right) \right)^k.
\end{align}

Recall the definition of $\deg_t(\alpha)$ given in \eqref{eq:deg_t}
Then the sequence $\alpha$ has exactly $\deg_t(\alpha)$ vertices that equals to $t$. Thus, we have the following identity
\begin{align*}
	\sum_{t=1}^s \deg_t(\alpha) = p.
\end{align*}

The following lemma computes the size of $\Delta_4(p,s;\alpha)$ for given $\alpha \in \cC_{s,p}^{(1)}$.

\begin{lemma} \label{lemma-Delta_4}
For fixed $\alpha = (\alpha_1, \ldots, \alpha_p) \in \cC_{s,p}^{(1)}$, the number of $i_1 \in \cC_{r,p}$ such that $g(i_1,\alpha) \in \Delta_4(p,s;\alpha)$ is
\begin{align} \label{eq-Delta_4 number}
	\sum_{t=1}^s \binom{\deg_t(\alpha)}{2}.
\end{align}
\end{lemma}

\begin{proof}
We sort the graphs in $\Delta_4(p,s;\alpha)$ into two classes. We will show that graphs in $\Delta_4(p,s;\alpha)$ can be transferred to graphs in $\Delta_1(p,s;\alpha)$ by splitting the vertex on the sequence $i_1$ associated to the multiple edges.

The first class consists of graphs whose the first $l-1$ up edges are distinct, and the $l$-th down edge $(\alpha_l,i_1^{(l)})$ coincide with the $j$-th down edge for some $1 \le j < l \le p$. Hence, $\alpha_l = \alpha_j$ and $i_1^{(l)} = i_1^{(j)}$. Since the subgraph of the path from $\alpha_j$ to $\alpha_l$ does not have a cycle, one can find $1 \le j' < l$, such that the $j'$-th up edge is from $i_1^{(l)}$ to $\alpha_l$. Another up edge from $i_1^{(l)}$ to $\alpha_l$ belongs to the set of the last $p-(l-1)$ up edges. Hence, the edge $(\alpha_l, i_1^{(l)})$ is not a down innovation. Now we split the vertex $i_1^{(l)}$ into two vertices $i_1^{(l,1)}$ and $i_1^{(l,2)}$. The new vertex $i_1^{(l,1)}$ is still labeled by the value of $i_1^{(l)}$ and the first $l-1$ pairs of up and down edges which connect $i_1^{(l)}$ are plotted to connect the new vertex $i_1^{(l,1)}$. The $l$-th down edge connect $i_1^{(l)}$ is plotted to connect the new vertex $i_1^{(l,2)}$, which is a down innovation. The rest of the edges can be plotted such that the new graph belongs to $\Delta_1(p,s;\alpha)$. See the graph below.

\begin{figure}[htbp]
  \centering
  \begin{minipage}{6.5in}
    \centering
    \raisebox{-0.5\height}{\includegraphics[scale=0.4]{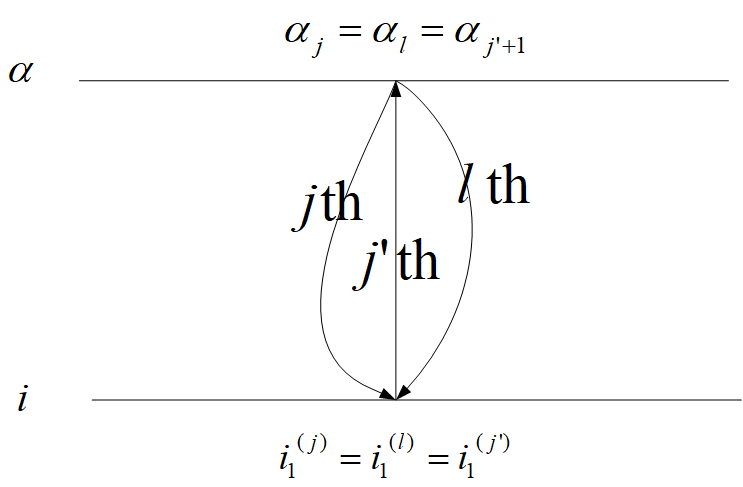}}
    \quad \raisebox{-0.5\height}{\mbox{$\boldsymbol{\longrightarrow}$}}\qquad
    \raisebox{-0.5\height}{\includegraphics[scale=0.4]{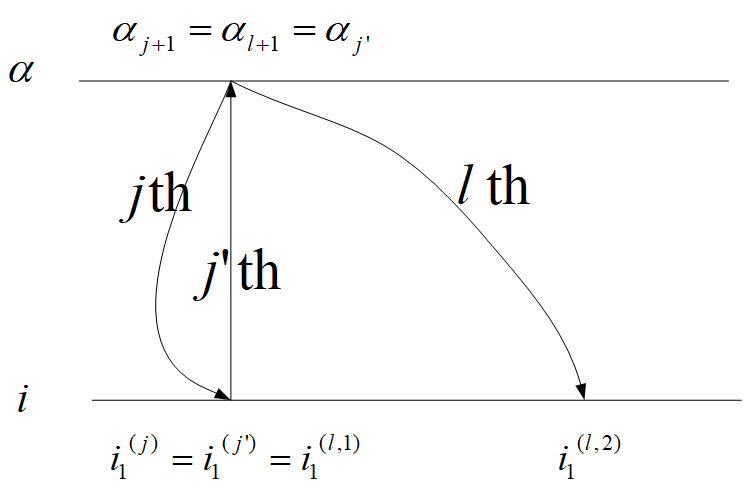}}
  \end{minipage}
  \caption{A graph in the first class of  $\Delta_4(p,s;\alpha)$
    (left column) and corresponding $\Delta_1(p,s;\alpha)$  graph
    (right column).\label{fig:Delta4_classe_1}}
\end{figure}

The second class consists of graphs whose first $l$ down edges are distinct, and the $l$-th up edge $(i_1^{(l)}, \alpha_{l+1})$ coincide with the $j$-th up edge for some $1 \le j < l \le p-1$. This means that $i_1^{(l)} = i_1^{(j)}$ and $\alpha_{l+1} = \alpha_{j+1}$. Since the subgraph of the path from $i_1^{(j)}$ to $i_1^{(l)}$ does not have a cycle, one can find $j < j' \le l$, such that the $j'$-th down edge is from $\alpha_{l+1}$ to $i_1^{(l)}$, which means that $\alpha_{j'} = \alpha_{l+1}$, $i_1^{(j')} = i_1^{(l)}$ and the $j'$-th down edge is not a down innovation. Another down edge from $\alpha_{l+1}$ to $i_1^{(l)}$ belongs to the set of the last $p-l$ down edges. Now we split the vertex $i_1^{(l)}$ into two vertices $i_1^{(l,1)}$ and $i_1^{(l,2)}$. The new vertex $i_1^{(l,1)}$ is still labeled by the value of $i_1^{(l)}$ and the first $j'-1$ pairs of up and down edges which connect $i_1^{(l)}$ are plotted to connect the new vertex $i_1^{(l,1)}$. The $j'$-th down edge connect $i_1^{(j')} = i_1^{(l)}$ is plotted to connect the new vertex $i_1^{(l,2)}$, which is a down innovation. The $l$-th up edge starts from $i_1^{(l,2)}$. The rest of the edges can be plotted such that the new graph belongs to $\Delta_1(p,s;\alpha)$. See the graph below.

\begin{figure}[htbp]
\centering
  \begin{minipage}{6.5in}
    \centering
    \raisebox{-0.5\height}{\includegraphics[scale=0.4]{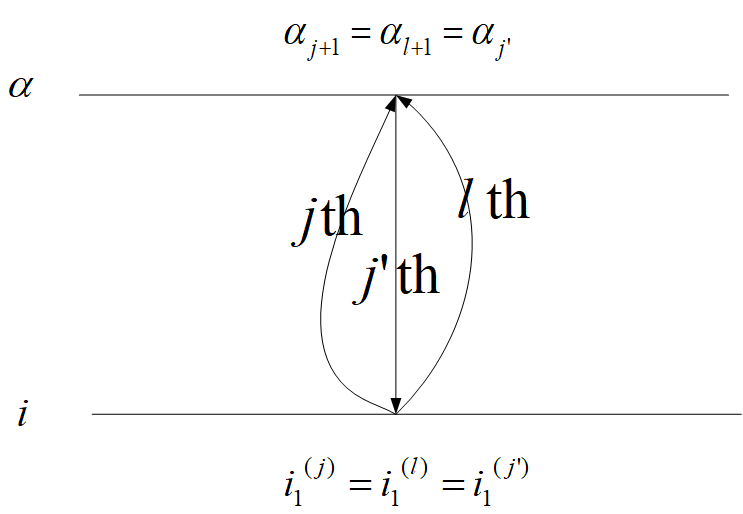}}
    \quad \raisebox{-0.5\height}{\mbox{$\boldsymbol{\longrightarrow}$}}\qquad
    \raisebox{-0.5\height}{\includegraphics[scale=0.4]{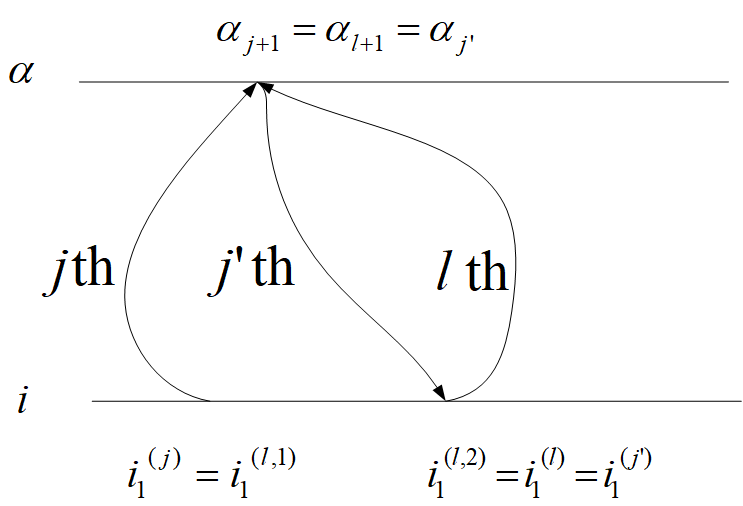}}
  \end{minipage}
  \caption{A graph in the second class of  $\Delta_4(p,s;\alpha)$
    (left column) and corresponding $\Delta_1(p,s;\alpha)$  graph
    (right column).\label{fig:Delta4_classe_2}}
\end{figure}

Therefore, for any graphs in $\Delta_4(p,s;\alpha)$, one can split the vertex on the sequence $i_1$ associated to the multiple edges to get a graph in $\Delta_1(p,s;\alpha)$. Equivalently, any graphs in $\Delta_4(p,s;\alpha)$ can be obtained from graphs in $\Delta_1(p,s;\alpha)$ by gluing two vertices on the sequence $i_1$ that has a same neighborhood. Moreover, gluing different pairs of vertices on the sequence $i_1$ leads to different graphs in $\Delta_4(p,s;\alpha)$.

For fixed $\alpha \in \cC_{s,p}^{(1)}$, by Lemma \ref{lemma-Bai}, there exists a unique sequence $i_1$, such that the graph $g(i_1,\alpha) \in \Delta_1(p,s;\alpha)$. For the vertex $t$ on the $\alpha$-line, an observation is that the number of its neighborhoods on the $i$-line is $\deg_t(\alpha)$. Hence, the choice to glue two vertices on the sequence $i_1$ with the same neighborhood $t$ on the $\alpha$-line is
\begin{align*}
	\binom{\deg_t(\alpha)}{2}.
\end{align*}
Thus, for fixed $\alpha \in \cC_{s,p}^{(1)}$, the number of $i_1 \in \cC_{r,p}$ such that $g(i_1,\alpha) \in \Delta_4(p,s;\alpha)$ is given by \eqref{eq-Delta_4 number}.
\end{proof}

The following lemma is a property of the sequence in $\cC_{s,p}^{(1)}$.

\begin{lemma} \label{lemma-abab}
For $\alpha \in \cC_{s,p}$, if there exist different integers $t_1, t_2 \in [s]$, such that $\alpha_{j_1} = \alpha_{j_1'} = t_1$ and $\alpha_{j_2} = \alpha_{j_2'} = t_2$ for some $1 \le j_1 < j_2 < j_1' < j_2' \le p$, then $\alpha \notin \cC_{s,p}^{(1)}$.
\end{lemma}

\begin{proof}
We prove by contradiction. Suppose not, then there exists a $w \in \cC_{p+1-s,p}$, such that $g(w,\alpha) \in \Delta_1(p,s;\alpha)$. Note that the path from $\alpha_{j_1}$ to $\alpha_{j_2}$ and the path from $\alpha_{j_1'}$ to $\alpha_{j_2'}$ are two paths from the vertex on the $\alpha$-line with labeled $t_1$ to the vertex on the $\alpha$-line with labeled $t_2$. The two paths will lead to multiple directed edges or undirected cycle, which is contradicted to the definition of $\Delta_1(p,s;\alpha)$.
\end{proof}

The following corollary is a direct consequence of Lemma \ref{lemma-abab}.

\begin{corollary}
For $\alpha \in \cC_{s,p}^{(1)}$, graphs in $\Delta_3(p,s;\alpha)$ are the $\Delta(p;\alpha)$-graphs in which each down edge must coincide with one and only one up edge. If we glue the coincident edges, the resulting graph is a connected graph with exactly one cycle.
\end{corollary}

The following lemma computes the size of $\Delta_3(p,s;\alpha)$ for given $\alpha \in \cC_{s,p}^{(1)}$.

\begin{lemma} \label{lemma-Delta_3}
For fixed $\alpha = (\alpha_1, \ldots, \alpha_p) \in \cC_{s,p}^{(1)}$, the number of $i_1 \in \cC_{r,p}$ such that $g(i_1,\alpha) \in \Delta_3(p,s;\alpha)$ is
\begin{align} \label{eq-Delta_3 number}
	\binom{p+1-s}{2} - \sum_{t=1}^s \binom{\deg_t(\alpha)}{2}.
\end{align}
\end{lemma}

\begin{proof}
Recall that down edges of graphs in $\Delta_3(p,s;\alpha)$ coincide with exactly one up edge, and the graph is connected with exactly one cycle when gluing coincident edges.

We also sort the graphs in $\Delta_3(p,s;\alpha)$ into two classes. We will show that graphs in $\Delta_3(p,s;\alpha)$ can be transferred to graphs in $\Delta_1(p,s;\alpha)$ by splitting an appropriate vertex on the sequence $i_1$.

The first class consists of graphs whose subgraph with first $l-1$ up and down edges has no cycle when orientation is removed and multiple edges are identified, but has a cycle when adding the $l$-th down edge. Hence, we have $3 \le l \le p$, and both $\alpha_l$ and $i_1^{(l)}$ are not new vertices. Now we split the vertex $i_1^{(l)}$ into two vertices $i_1^{(l,1)}$ and $i_1^{(l,2)}$. The new vertex $i_1^{(l,1)}$ is still labeled by the value of $i_1^{(l)}$, and the first $l-1$ pairs of up and down edges which connect $i_1^{(l)}$ are plotted to connect the new vertex $i_1^{(l,1)}$. The $l$-th down edge $(\alpha_l, i_1^{(l)})$ is replaced by the new down edge $(\alpha_l, i_1^{(l,2)})$, which is a down innovation. The $l$-th up edge starts from $i_1^{(l,2)}$. The rest of the edges can be plotted such that the new graph belongs to $\Delta_1(p,s;\alpha)$. See the graph below.

\begin{figure}[htbp]
  \begin{minipage}{6.5in}
    \centering
    \raisebox{-0.5\height}{\includegraphics[scale=0.4]{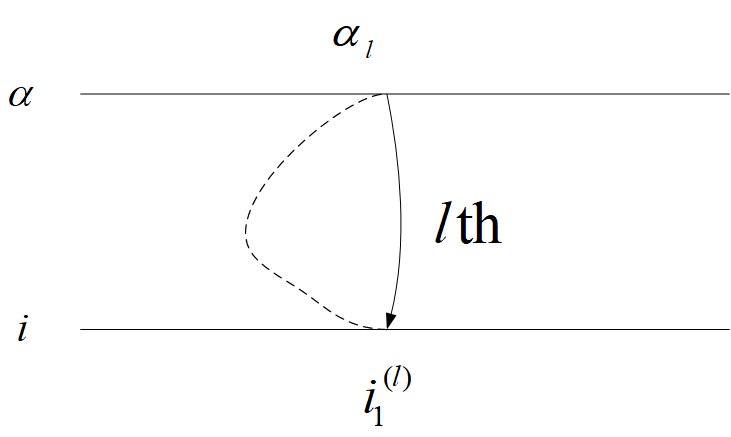}}
    \quad \raisebox{-0.5\height}{\mbox{$\boldsymbol{\longrightarrow}$}}\qquad
    \raisebox{-0.5\height}{\includegraphics[scale=0.4]{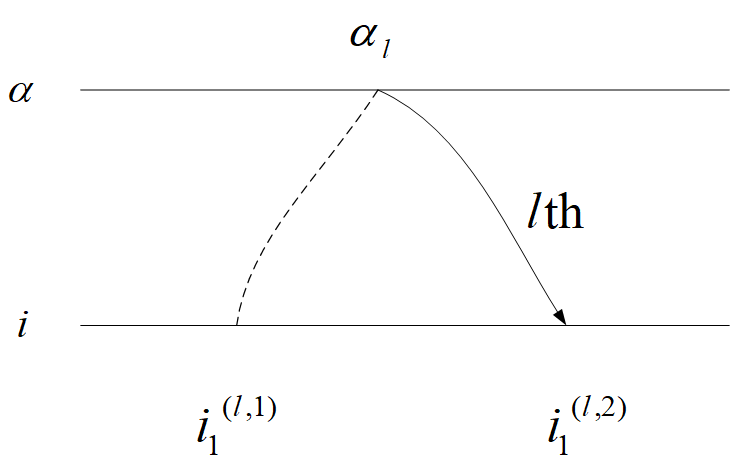}}
  \end{minipage}
  \caption{A graph in the first class of  $\Delta_3(p,s;\alpha)$
    (left column) and corresponding $\Delta_1(p,s;\alpha)$  graph
    (right column). The dashed line describes the existence of a cycle.\label{fig:Delta3_class_1}}
\end{figure}

The second class consists of graphs whose subgraph with first $l-1$ up edges and first $l$ down edges has no cycle when orientation is removed and multiple edges are identified, but has a cycle when adding the $l$-th up edge. Thus, there exists a $1 \le j < l$, such that $\alpha_{l+1} = \alpha_j$ and the $j$-th down edge belongs to the cycle. Since the subgraph with first $l-1$ up edges and first $l$ down edges has no cycle, we have $\alpha_u \neq \alpha_j$ for all $j < u < l+1$. By Lemma \ref{lemma-abab}, the first $j-1$ pairs as well as the last $p-l$ pairs of down and up edges never visit the vertex $\alpha_u$ for all $j < u < l+1$. Hence, along the path from $i_1^{(j)}$ to $i_1^{(l)}$, all down edges are paired with exactly one up edge. Note that along the path from $\alpha_j$ to $i_1^{(l)}$, there must be a single edge, which can only be $(\alpha_j, i_1^{(j)})$. If $i_1^{(j)} \neq i_1^{(l)}$, then in the subgraph of path from $\alpha_j$ to $\alpha_{l+1}$, the degree of the vertices $i_1^{(j)}$ and $i_1^{(l)}$ are both odd number. This is impossible. However, $i_1^{(j)} = i_1^{(l)}$ implies that the $l$-th up edge $(i_1^{(l)}, \alpha_{l+1})$ coincide with the $j$-th down edge $(\alpha_j, i_1^{(j)})$. This contradicts to the assumption that the subgraph with first $l$ pairs of down and up edges has a cycle but has no cycle if deleting the $l$-th up edges, when the orientation is removed and multiple edges are identified. Thus, the second class is empty.

Therefore, for any graphs in $\Delta_3(p,s;\alpha)$, one can split an appropriate vertex on the sequence $i_1$ to get a graph in $\Delta_1(p,s;\alpha)$. Equivalently, any graphs in $\Delta_3(p,s;\alpha)$ can be obtained from graphs in $\Delta_1(p,s;\alpha)$ by gluing two vertices on the sequence $i_1$ without common neighborhood. Moreover, gluing different pairs of vertices on the sequence $i_1$ leads to different graphs in $\Delta_3(p,s;\alpha)$.

For fixed $\alpha \in \cC_{s,p}^{(1)}$, by Lemma \ref{lemma-Bai}, there exists a unique sequence $i_1$, such that the graph $g(i_1,\alpha) \in \Delta_1(p,s;\alpha)$. Moreover, the sequence $i_1$ has $p+1-s$ different vertices. Hence, the choice to glue two vertices on the sequence $i_1$ is $\binom{p+1-s}{2}$. By Lemma \ref{lemma-Delta_4}, the choice to glue two vertices on the sequence $i_1$ that has different neighborhoods is given by \eqref{eq-Delta_3 number}.
\end{proof}

By Lemma \ref{lemma-Delta_4} ,\ref{lemma-Delta_3} and \eqref{eq-d>0-moment-2.2}, we have
\begin{align} \label{eq-2.6}
	& \dfrac{1}{n^k} \bE \big[ \Tr M_{n,k,m}^p \big] \nonumber \\
	=& \sum_{s=1}^p \left( \dfrac{m}{n^k} \right)^s (1+o(1)) \sum_{\alpha \in \cC_{s,p}^{(1)}} \left( \sum_{\varphi \in \cI_{s,m}} \prod_{t=1}^s \dfrac{1}{m} \tau_{\varphi(t)}^{\deg_t(\alpha)} \right) \nonumber \\
	& \left[ 1 - \dfrac{(p+1-s)(p-s)}{2n} + \dfrac{1}{n} \left( \binom{p+1-s}{2} - \sum_{t=1}^s \binom{\deg_t(\alpha)}{2} + m_4 \sum_{t=1}^s \binom{\deg_t(\alpha)}{2} \right) \right]^k \nonumber \\
	=& \sum_{s=1}^p \left( \dfrac{m}{n^k} \right)^s (1+o(1)) \sum_{\alpha \in \cC_{s,p}^{(1)}} \left( \sum_{\varphi \in \cI_{s,m}} \prod_{t=1}^s \dfrac{1}{m} \tau_{\varphi(t)}^{\deg_t(\alpha)} \right) \left[ 1 + \dfrac{m_4-1}{n} \sum_{t=1}^s \binom{\deg_t(\alpha)}{2} \right]^k \nonumber \\
	=& \sum_{s=1}^p \left( \dfrac{m}{n^k} \right)^s (1+o(1)) \sum_{\alpha \in \cC_{s,p}^{(1)}} \prod_{t=1}^s \left( \dfrac{1}{m} \sum_{j=1}^m \tau_j^{\deg_t(\alpha)} \right) \left[ 1 + \dfrac{m_4-1}{n} \sum_{t=1}^s \binom{\deg_t(\alpha)}{2} \right]^k,
\end{align}
where we use the following equality in the last equality
\begin{align*}
	\sum_{\varphi \in \cI_{s,m}} \prod_{t=1}^s \dfrac{1}{m} \tau_{\varphi(t)}^{\deg_t(\alpha)}
	= \sum_{j_1, \ldots, j_s=1 \atop \forall p \neq q, j_p \neq j_q}^m \prod_{t=1}^s \dfrac{1}{m} \tau_{j_t}^{\deg_t(\alpha)}
	= \prod_{t=1}^s \left( \dfrac{1}{m} \sum_{j=1}^m \tau_j^{\deg_t(\alpha)} \right) (1+o(1)).
\end{align*}
Letting $m,n,k \to \infty$ in \eqref{eq-2.6} with the ratio \eqref{eq-def-ratio} and the limit assumption \eqref{eq-limit of tau}, we have
\begin{align} 
	\lim_{n \to \infty} \dfrac{1}{n^k} \bE \big[ \Tr M_{n,k,m}^p \big]
	=& \sum_{s=1}^p c^s \sum_{\alpha \in \cC_{s,p}^{(1)}} \left( \prod_{t=1}^s m_{\deg_t(\alpha)}^{(\tau)} \right) \exp \left( d(m_4-1) \sum_{t=1}^s \binom{\deg_t(\alpha)}{2} \right).
\end{align}


%
\subsection{Proof of (ii)}\label{sec:part(ii)}

For any $p \in \bN$, we compute the variance of the moment
\begin{align} \label{eq-3.1-variance}
	\Var \left( \dfrac{1}{n^k} \Tr M_{n,k,m}^p \right)
	= \dfrac{1}{n^{2k}} \bE \left[ \big( \Tr M_{n,k,m}^p \big)^2 \right]
	- \dfrac{1}{n^{2k}} \left( \bE \left[ \Tr M_{n,k,m}^p \right] \right)^2.
\end{align}
We use the convention that $\bi^{(p+1)} = \bi^{(1)}$ and $\j^{(p+1)} = \j^{(1)}$. Recalling the model \eqref{eq:matrix}, we can write
\begin{align*}
    & \dfrac{1}{n^{2k}} \bE \left[ \big( \Tr M_{n,k,m}^p \big)^2 \right] \nonumber \\
	=& \dfrac{1}{n^{2k}} \sum_{\alpha_1, \ldots, \alpha_p = 1}^m \sum_{\beta_1, \ldots, \beta_p = 1}^m \left( \prod_{t=1}^p \tau_{\alpha_t} \tau_{\beta_t} \right)
	\bE \Big[ \Tr \big( Y_{\alpha_1} Y_{\alpha_1}^* \cdots Y_{\alpha_p} Y_{\alpha_p}^* \big) \Tr \big( Y_{\beta_1} Y_{\beta_1}^* \cdots Y_{\beta_p} Y_{\beta_p}^* \big) \Big] \nonumber \\
	=& \dfrac{1}{n^{2k}} \sum_{\alpha_1, \ldots, \alpha_p = 1}^m \sum_{\beta_1, \ldots, \beta_p = 1}^m \left( \prod_{t=1}^p \tau_{\alpha_t} \tau_{\beta_t} \right)
	\bE \Bigg[ \sum_{\bi^{(1)}, \ldots, \bi^{(p)} = 1}^{n^k} \sum_{\j^{(1)}, \ldots, \j^{(p)} = 1}^{n^k} \prod_{t=1}^p Y_{\bi^{(t)} \alpha_t} \overline{Y_{\bi^{(t+1)} \alpha_t}} Y_{\j^{(t)} \beta_t} \overline{Y_{\j^{(t+1)} \beta_t}} \Bigg] \nonumber \\
	=& \dfrac{1}{n^{2k}} \sum_{\alpha_1, \ldots, \alpha_p = 1}^m \sum_{\beta_1, \ldots, \beta_p = 1}^m \left( \prod_{t=1}^p \tau_{\alpha_t} \tau_{\beta_t} \right) \nonumber \\
	& \times \bE \Bigg[ \sum_{\bi^{(1)}, \ldots, \bi^{(p)} = 1}^{n^k} \sum_{\j^{(1)}, \ldots, \j^{(p)} = 1}^{n^k} \prod_{t=1}^p \prod_{l=1}^k \bigg( \big( \y_{\alpha_t}^{(l)} \big)_{i^{(t)}_l} \overline{\big( \y_{\alpha_t}^{(l)} \big)_{i^{(t+1)}_l}} \big( \y_{\beta_t}^{(l)} \big)_{j^{(t)}_l} \overline{\big( \y_{\beta_t}^{(l)} \big)_{j^{(t+1)}_l}} \bigg) \Bigg].
\end{align*}
The terms in the last line can be factored in the following way:
\begin{align} \label{eq-3.2-second moment}
	& \dfrac{1}{n^{2k}} \bE \left[ \big( \Tr M_{n,k,m}^p \big)^2 \right] \nonumber \\
	=& \dfrac{1}{n^{2k}} \sum_{\alpha_1, \ldots, \alpha_p = 1}^m \sum_{\beta_1, \ldots, \beta_p = 1}^m \left( \prod_{t=1}^p \tau_{\alpha_t} \tau_{\beta_t} \right) \nonumber \\
	&\times \bE \Bigg[ \prod_{l=1}^k \sum_{i_l^{(1)}, \ldots, i_l^{(p)} = 1}^n \sum_{j_l^{(1)}, \ldots, j_l^{(p)} = 1}^n \prod_{t=1}^p \bigg( \big( \y_{\alpha_t}^{(l)} \big)_{i^{(t)}_l} \overline{\big( \y_{\alpha_t}^{(l)} \big)_{i^{(t+1)}_l}} \big( \y_{\beta_t}^{(l)} \big)_{j^{(t)}_l} \overline{\big( \y_{\beta_t}^{(l)} \big)_{j^{(t+1)}_l}} \bigg) \Bigg] \nonumber \\
	=& \dfrac{1}{n^{2k}} \sum_{\alpha_1, \ldots, \alpha_p = 1}^m \sum_{\beta_1, \ldots, \beta_p = 1}^m \left( \prod_{t=1}^p \tau_{\alpha_t} \tau_{\beta_t} \right) \nonumber \\
	&\times \left( \bE \left[ \sum_{i_1^{(1)}, \ldots, i_1^{(p)} = 1}^n \sum_{j_1^{(1)}, \ldots, j_1^{(p)} = 1}^n \prod_{t=1}^p \bigg( \big( \y_{\alpha_t}^{(1)} \big)_{i^{(t)}_1} \overline{\big( \y_{\alpha_t}^{(1)} \big)_{i^{(t+1)}_1}} \big( \y_{\beta_t}^{(1)} \big)_{j^{(t)}_1} \overline{\big( \y_{\beta_t}^{(1)} \big)_{j^{(t+1)}_1}} \bigg) \right] \right)^k.
\end{align}
Here, we use the i.i.d. setting in the last equality. By \eqref{eq-moment-0.2}, we have
\begin{align} \label{eq-3.3-E[]^2}
	& \dfrac{1}{n^{2k}} \left( \bE \big[ \Tr M_{n,k,m}^p \big] \right)^2 \nonumber \\
	=& \dfrac{1}{n^{2k}} \sum_{\alpha_1, \ldots, \alpha_p = 1}^m \sum_{\beta_1, \ldots, \beta_p = 1}^m \left( \prod_{t=1}^p \tau_{\alpha_t} \tau_{\beta_t} \right) \nonumber \\
	&\times \left( \bE \left[ \sum_{i_1^{(1)}, \ldots, i_1^{(p)} = 1}^n \prod_{t=1}^p \bigg( \big( \y_{\alpha_t}^{(1)} \big)_{i^{(t)}_1} \overline{\big( \y_{\alpha_t}^{(1)} \big)_{i^{(t+1)}_1}} \bigg) \right] \bE \left[ \sum_{j_1^{(1)}, \ldots, j_1^{(p)} = 1}^n \prod_{t=1}^p \bigg( \big( \y_{\beta_t}^{(1)} \big)_{j^{(t)}_1} \overline{\big( \y_{\beta_t}^{(1)} \big)_{j^{(t+1)}_1}} \bigg) \right] \right)^k.
\end{align}
Substitute \eqref{eq-3.2-second moment} and \eqref{eq-3.3-E[]^2} to \eqref{eq-3.1-variance}, we have
\begin{align} \label{eq-3.4-variance}
	& \Var \left( \dfrac{1}{n^k} \Tr M_{n,k,m}^p \right) \nonumber \\
	=& \dfrac{1}{n^{2k}} \sum_{\alpha_1, \ldots, \alpha_p = 1}^m \sum_{\beta_1, \ldots, \beta_p = 1}^m \left( \prod_{t=1}^p \tau_{\alpha_t} \tau_{\beta_t} \right) \nonumber \\
	&\times \left\{ \left( \bE \left[ \sum_{i_1^{(1)}, \ldots, i_1^{(p)} = 1}^n \sum_{j_1^{(1)}, \ldots, j_1^{(p)} = 1}^n \prod_{t=1}^p \bigg( \big( \y_{\alpha_t}^{(1)} \big)_{i^{(t)}_1} \overline{\big( \y_{\alpha_t}^{(1)} \big)_{i^{(t+1)}_1}} \big( \y_{\beta_t}^{(1)} \big)_{j^{(t)}_1} \overline{\big( \y_{\beta_t}^{(1)} \big)_{j^{(t+1)}_1}} \bigg) \right] \right)^k \right. \nonumber \\
	& \left. - \left( \bE \left[ \sum_{i_1^{(1)}, \ldots, i_1^{(p)} = 1}^n \prod_{t=1}^p \bigg( \big( \y_{\alpha_t}^{(1)} \big)_{i^{(t)}_1} \overline{\big( \y_{\alpha_t}^{(1)} \big)_{i^{(t+1)}_1}} \bigg) \right]
	\bE \left[ \sum_{j_1^{(1)}, \ldots, j_1^{(p)} = 1}^n \prod_{t=1}^p \bigg( \big( \y_{\beta_t}^{(1)} \big)_{j^{(t)}_1} \overline{\big( \y_{\beta_t}^{(1)} \big)_{j^{(t+1)}_1}} \bigg) \right] \right)^k \right\}.
\end{align}
Denote the sequences $\alpha = (\alpha_1, \ldots, \alpha_p)$, $\beta = (\beta_1, \ldots, \ldots, \beta_p)$, $i_1 = (i_1^{(1)}, \ldots, i_1^{(p)})$ and $j_1 = (j_1^{(1)}, \ldots, j_1^{(p)})$. We use the notation $\alpha \cap \beta = \emptyset$ if the two sequences $\alpha$ and $\beta$ have no common vertices. Note that when $\alpha \cap \beta = \emptyset$, the two random variables
\begin{align*}
	\prod_{t=1}^p \bigg( \big( \y_{\alpha_t}^{(1)} \big)_{i^{(t)}_1} \overline{\big( \y_{\alpha_t}^{(1)} \big)_{i^{(t+1)}_1}} \bigg),
	~ \mathrm{and} ~
	\prod_{t=1}^p \bigg( \big( \y_{\beta_t}^{(1)} \big)_{j^{(t)}_1} \overline{\big( \y_{\beta_t}^{(1)} \big)_{j^{(t+1)}_1}} \bigg)
\end{align*}
are independent, for any sequences $i_1, j_1 \in [n]^p$. Hence, \eqref{eq-3.4-variance} can be written as
\begin{align} \label{eq-3.5-variance}
	& \Var \left( \dfrac{1}{n^k} \Tr M_{n,k,m}^p \right) \nonumber \\
	=& \dfrac{1}{n^{2k}} \sum_{\alpha, \beta \in [m]^p, \alpha \cap \beta \neq \emptyset} \left( \prod_{t=1}^p \tau_{\alpha_t} \tau_{\beta_t} \right) \nonumber \\
	&\times \left\{ \left( \sum_{i_1, j_1 \in [n]^p} \bE \left[ \prod_{t=1}^p \bigg( \big( \y_{\alpha_t}^{(1)} \big)_{i^{(t)}_1} \overline{\big( \y_{\alpha_t}^{(1)} \big)_{i^{(t+1)}_1}} \big( \y_{\beta_t}^{(1)} \big)_{j^{(t)}_1} \overline{\big( \y_{\beta_t}^{(1)} \big)_{j^{(t+1)}_1}} \bigg) \right] \right)^k \right. \nonumber \\
	& \left. - \left( \sum_{i_1, j_1 \in [n]^p} \bE \left[ \prod_{t=1}^p \bigg( \big( \y_{\alpha_t}^{(1)} \big)_{i^{(t)}_1} \overline{\big( \y_{\alpha_t}^{(1)} \big)_{i^{(t+1)}_1}} \bigg) \right] \bE \left[ \prod_{t=1}^p \bigg( \big( \y_{\beta_t}^{(1)} \big)_{j^{(t)}_1} \overline{\big( \y_{\beta_t}^{(1)} \big)_{j^{(t+1)}_1}} \bigg) \right] \right)^k \right\} \nonumber \\
	=& \dfrac{1}{n^{2k}} \sum_{\alpha, \beta \in [m]^p, \alpha \cap \beta \neq \emptyset} \left( \prod_{t=1}^p \tau_{\alpha_t} \tau_{\beta_t} \right) \nonumber \\
	&\times \left\{ \left( \sum_{i_1, j_1 \in [n]^p} E'(i_1,\alpha;j_1,\beta) \right)^k - \left( \sum_{i_1, j_1 \in [n]^p} E(i_1,\alpha) E(j_1,\beta) \right)^k \right\},
\end{align}
where $E(\cdot,\cdot)$ is defined in \eqref{eq-def-E(i_1,alpha)}, and $E'(i_1,\alpha;j_1,\beta)$ is given by
\begin{align*}
	E'(i_1,\alpha;j_1,\beta)
	= \bE \left[ \prod_{t=1}^p \bigg( \big( \y_{\alpha_t}^{(1)} \big)_{i^{(t)}_1} \overline{\big( \y_{\alpha_t}^{(1)} \big)_{i^{(t+1)}_1}} \big( \y_{\beta_t}^{(1)} \big)_{j^{(t)}_1} \overline{\big( \y_{\beta_t}^{(1)} \big)_{j^{(t+1)}_1}} \bigg) \right].
\end{align*}

We use the notation $g(i_1,\alpha) \cup g(j_1,\beta)$ for the graph obtained from $g(i_1,\alpha)$ and $g(j_1,\beta)$ by joining two graphs together and keep the coincident edges. If the graph $g(i_1,\alpha) \cup g(j_1,\beta)$ has a single edge, then this single edge must belong to one of the graphs $g(i_1,\alpha)$ and $g(j_1,\beta)$. Since the variables have mean zero, only the summation indices $i_1, j_1$ such that the graph $g(i_1,\alpha) \cup g(j_1,\beta)$ has no single edge contribute to \eqref{eq-3.5-variance}. Note that the graph $g(i_1,\alpha) \cup g(j_1,\beta)$ is connected with $4p$ edges, the degrees of the vertices is at least $4$ with at most two exceptions, whose degrees are at least $2$. Denote $s = |(\alpha, \beta)|$ and $r = |(i_1, j_1)|$, then the term $E'(i_1,\alpha;j_1,\beta)$ is non-vanishing when $r+s \le 2p+1$.

Besides, if $\alpha \cap \beta \neq \emptyset$, then $s \le |\alpha| + |\beta| - 1$. Hence, the term $E(i_1,\alpha) E(j_1, \beta)$ is non-vanishing when $|i_1| + |\alpha| \le p+1$ and $|j_1| + |\beta| \le p+1$, which implies that $r+s \le |i_1| + |j_1| + (|\alpha| + |\beta| -1) \le 2p+1$. Hence, for $n$ large,
\begin{align} \label{eq-3.6-variance}
	& \Var \left( \dfrac{1}{n^k} \Tr M_{n,k,m}^p \right) \nonumber \\
	=& \dfrac{1}{n^{2k}} \sum_{s=1}^{2p} \sum_{(\alpha, \beta) \in \cC_{s,2p}, \alpha \cap \beta \neq \emptyset} \left( \sum_{\varphi \in \cI_{s,m}} \prod_{t=1}^p \tau_{\varphi(\alpha_t)} \tau_{\varphi(\beta_t)} \right) \nonumber \\
	&\times \left\{ \left( \sum_{r=1}^{2p} n^r(1+O(n^{-1})) \sum_{(i_1, j_1) \in \cC_{r,2p}} E'(i_1,\alpha;j_1,\beta) \right)^k \right. \nonumber \\
	&\qquad \left. - \left( \sum_{r=1}^{2p} n^r(1+O(n^{-1})) \sum_{(i_1, j_1) \in \cC_{r,2p}} E(i_1,\alpha) E(j_1,\beta) \right)^k \right\} \nonumber \\
	\le& \dfrac{1}{n^{2k}} \sum_{s=1}^{2p} \sum_{(\alpha, \beta) \in \cC_{s,2p}, \alpha \cap \beta \neq \emptyset} \left( \sum_{\varphi \in \cI_{s,m}} \prod_{t=1}^s \tau_{\varphi(t)}^{\deg_t(\alpha) + \deg_t(\beta)} \right) \times C_p^k n^{(1-s)k} \nonumber \\
	=& \dfrac{C_p^k}{n^k} \sum_{s=1}^{2p} \left( \dfrac{m}{n^k} \right)^s \sum_{(\alpha, \beta) \in \cC_{s,2p}, \alpha \cap \beta \neq \emptyset} \left( \sum_{\varphi \in \cI_{s,m}} \prod_{t=1}^s \dfrac{1}{m} \tau_{\varphi(t)}^{\deg_t(\alpha) + \deg_t(\beta)} \right) \nonumber \\
	\le& \dfrac{C_p^k}{n^k} \sum_{s=1}^{2p} \left( \dfrac{m}{n^k} \right)^s \sum_{(\alpha, \beta) \in \cC_{s,2p}, \alpha \cap \beta \neq \emptyset} \prod_{t=1}^s \left( \dfrac{1}{m} \sum_{j=1}^m |\tau_j|^{\deg_t(\alpha) + \deg_t(\beta)} \right) \nonumber \\
	\le& \dfrac{C_p^k}{n^k}.
\end{align}
where $C_p$ is a positive number that depends only on $p$ and may varies in different places. Therefore, for all fixed $p \in \bN_+$, if $k \ge 2$, we have
\begin{align*}
	\sum_{n=1}^{\infty} \Var \left( \dfrac{1}{n^k} \Tr M_{n,k,m}^p \right) < \infty.
\end{align*}

\appendix

\section{The case \texorpdfstring{$d=0$}{d=0}.}
\label{sec:d=0}

The variance calculation (ii) is identical to the one given in the proof of Theorem~\ref{th:moment}, see Section~\ref{sec:part(ii)}. It remains to establish the moment limit (i).

Note that the moment formula \eqref{eq-moment-0.4} is also valid for the case $d=0$. We use the graph $g(i_1,\alpha)$ to compute the sum on $i_1$ in \eqref{eq-moment-0.4}. For any sequences $\alpha \in \cC_{s,p}$ and $i_1 \in \cC_{r,p}$ satisfying $g(i_1,\alpha) \in \Delta_1(p,s;\alpha)$, we have $r+s=p+1$ and $E(i_1,\alpha) = n^{-p}$. Besides, $E(i_1,\alpha)=0$ if the graph $g(i_1,\alpha) \in \Delta_2(p,r,s;\alpha)$. Moreover, if $g(i_1,\alpha) \notin \Delta_1(p,s;\alpha) \cup \Delta_2(p,r,s;\alpha)$, then $r+s \le p$ and $E(i_1,\alpha) = O(n^{-p})$. Hence,
\begin{align} \label{eq-d=0-sum i_1}
	& \sum_{r=1}^p n \cdots (n-r+1) \sum_{i_1 \in \cC_{r,p}} E(i_1,\alpha) \nonumber \\
	=& \sum_{r=1}^p n^r \big( 1+O(n^{-1}) \big) \sum_{i_1 \in \cC_{r,p}} E(i_1,\alpha) \nonumber \\
	=& \sum_{r=1}^p n^r \big( 1+O(n^{-1}) \big) \sum_{i_1 \in \cC_{r,p} \atop g(i_1,\alpha) \in \Delta_1(p,s;\alpha)} E(i_1,\alpha) \nonumber \\
	&+ \sum_{r=1}^p n^r (1+O(n^{-1})) \sum_{i_1 \in \cC_{r,p} \atop g(i_1,\alpha) \notin \Delta_1(p,s;\alpha) \cup \Delta_2(p,r,s;\alpha)} E(i_1,\alpha) \nonumber \\
	=& n^{1-s} \big( 1+O(n^{-1}) \big) \sum_{i_1 \in \cC_{p+1-s,p} \atop g(i_1,\alpha) \in \Delta_1(p,s;\alpha)} 1
	+ O(n^{-s}).
\end{align}

Therefore, by \eqref{eq-moment-0.4}, \eqref{eq-d=0-sum i_1} and Lemma \ref{lemma-Bai}, we have
\begin{align} \label{eq-d=0-moment}
	\dfrac{1}{n^k} \bE \big[ \Tr M_{n,k,m}^p \big]
	=& \dfrac{1}{n^k} \sum_{s=1}^p \sum_{\alpha \in \cC_{s,p}} \left( \sum_{\varphi \in \cI_{s,m}} \prod_{t=1}^p \tau_{\varphi(\alpha_t)} \right) \left( n^{1-s} \big( 1+O(n^{-1}) \big) \sum_{i_1 \in \cC_{p+1-s,p} \atop g(i_1,\alpha) \in \Delta_1(p,s;\alpha)} 1 + O(n^{-s}) \right)^k \nonumber \\
	=& \dfrac{1}{n^k} \sum_{s=1}^p \sum_{\alpha \in \cC_{s,p}^{(1)}} \left( \sum_{\varphi \in \cI_{s,m}} \prod_{t=1}^p \tau_{\varphi(\alpha_t)} \right) \Big( n^{1-s} \big( 1+O(n^{-1}) \big) + O(n^{-s}) \Big)^k \nonumber \\
	& + \dfrac{1}{n^k} \sum_{s=1}^p \sum_{\alpha \in \cC_{s,p} \setminus \cC_{s,p}^{(1)}} \left( \sum_{\varphi \in \cI_{s,m}} \prod_{t=1}^p \tau_{\varphi(\alpha_t)} \right) \big( O(n^{-s}) \big)^k \nonumber \\
	=& \dfrac{1}{n^k} \sum_{s=1}^p \sum_{\alpha \in \cC_{s,p}^{(1)}} \left( \sum_{\varphi \in \cI_{s,m}} \prod_{t=1}^p \tau_{\varphi(\alpha_t)} \right) \times n^{k(1-s)} (1+o(1)) \nonumber \\
	=& \sum_{s=1}^p \left( \dfrac{m}{n^k} \right)^s \sum_{\alpha \in \cC_{s,p}^{(1)}} \left( \dfrac{1}{m^s} \sum_{\varphi \in \cI_{s,m}} \prod_{t=1}^p \tau_{\varphi(\alpha_t)} \right) (1+o(1)) \nonumber \\
	=& \sum_{s=1}^p \left( \dfrac{m}{n^k} \right)^s \sum_{\alpha \in \cC_{s,p}^{(1)}} \left( \prod_{t=1}^s \left( \dfrac{1}{m} \sum_{j=1}^m \tau_j^{\deg_t(\alpha)} \right) \right) (1+o(1)),
\end{align}
where we use $k = o(n)$ in the third equality. Letting $m,n,k \to \infty$ with the ratio \eqref{eq-def-ratio}, we have
\begin{align*}
	\lim_{n \to \infty} \dfrac{1}{n^k} \bE \big[ \Tr M_{n,k,m}^p \big]
	= \sum_{s=1}^p c^s \sum_{\alpha \in \cC_{s,p}^{(1)}} \left( \prod_{t=1}^s m_{\deg_t(\alpha)}^{(\tau)} \right).
\end{align*}

\begin{remark}
In the case $\tau_1 = \tau_2 = \cdots = 1$, by Lemma \ref{lemma-Bai}, we have
\begin{align*}
	\lim_{n \to \infty} \dfrac{1}{n^k} \bE \big[ \Tr M_{n,k,m}^p \big]
	= \sum_{s=1}^p \dfrac{1}{p} \binom{p}{s-1} \binom{p}{s} c^s,
\end{align*}
which is the $p$-th moment of the Marchenko-Pastur law.
\end{remark}

Acknowledgments: 
BC was partially supported by JSPS Kakenhi 17H04823, 20K20882, 21H00987, JPJSBP120203202.

\bibliographystyle{plain}
\bibliography{tensor}

\begin{bibdiv}
\begin{biblist}

\bib{Hastings2012}{article}{
      author={Ambainis, Andris},
      author={Harrow, Aram~W.},
      author={Hastings, Matthew~B.},
       title={Random tensor theory: extending random matrix theory to mixtures
  of random product states},
        date={2012},
        ISSN={0010-3616},
     journal={Comm. Math. Phys.},
      volume={310},
      number={1},
       pages={25\ndash 74},
         url={https://doi.org/10.1007/s00220-011-1411-x},
      review={\MR{2885613}},
}

\bib{BaiZhou08}{article}{
      author={Bai, Z.},
      author={Zhou, W.},
       title={Large sample covariance matrices without independence structures
  in columns},
        date={2008},
     journal={Statist. Sinica},
      volume={18},
      number={2},
       pages={425\ndash 442},
}

\bib{Bai2010}{book}{
      author={Bai, Zhidong},
      author={Silverstein, Jack~W.},
       title={Spectral analysis of large dimensional random matrices},
     edition={Second},
      series={Springer Series in Statistics},
   publisher={Springer, New York},
        date={2010},
        ISBN={978-1-4419-0660-1},
         url={https://doi.org/10.1007/978-1-4419-0661-8},
      review={\MR{2567175}},
}

\bib{Collins2013}{article}{
      author={Collins, Beno\^{\i}t},
      author={Nechita, Ion},
      author={\.{Z}yczkowski, Karol},
       title={Area law for random graph states},
        date={2013},
        ISSN={1751-8113},
     journal={J. Phys. A},
      volume={46},
      number={30},
       pages={305302, 18},
         url={https://doi.org/10.1088/1751-8113/46/30/305302},
      review={\MR{3083280}},
}

\bib{Lin2017}{article}{
      author={Lin, G.D.},
       title={Recent developments on the moment problem},
        date={2017},
     journal={Journal of Statistical Distributions and Applications},
      volume={4},
      number={1},
}

\bib{Lytova2018}{article}{
      author={Lytova, A.},
       title={Central limit theorem for linear eigenvalue statistics for a
  tensor product version of sample covariance matrices},
        date={2018},
        ISSN={0894-9840},
     journal={J. Theoret. Probab.},
      volume={31},
      number={2},
       pages={1024\ndash 1057},
         url={https://doi.org/10.1007/s10959-017-0741-9},
      review={\MR{3803923}},
}

\bib{LytovaPastur09}{article}{
      author={Lytova, A.},
      author={Pastur, L.},
       title={Central limit theorem for linear eigenvalue statistics of the
  {W}igner and the sample covariance random matrices},
        date={2009},
     journal={Ann. Probab.},
      volume={37},
       pages={1778\ndash 1840},
}

\bib{MP67}{article}{
      author={Mar\v{c}enko, V.A.},
      author={Pastur, L.A.},
       title={Distribution of eigenvalues for some sets of random matrices},
        date={1967},
     journal={Math. USSR-Sb},
      volume={1},
       pages={457\ndash 483},
}

\bib{Pajor09}{article}{
      author={Pajor, A.},
      author={Pastur, L.},
       title={On the limiting empirical measure of eigenvalues of the sum of
  rank one matrices with log-concave distribution},
        date={2009},
     journal={Studia Mathematica},
      volume={195},
      number={1},
       pages={11\ndash 29},
}

\bib{Silv95}{article}{
      author={Silverstein, Jack~W.},
       title={Strong convergence of the empirical distribution of eigenvalues
  of large-dimensional random matrices},
        date={1995},
     journal={J. Multivariate Anal.},
      volume={55},
      number={2},
       pages={331\ndash 339},
}

\bib{Yaskov2016}{article}{
      author={Yaskov, Pavel},
       title={Necessary and sufficient conditions for the {M}archenko-{P}astur
  theorem},
        date={2016},
     journal={Electron. Commun. Probab.},
      volume={21},
       pages={Paper No. 73, 8},
         url={https://doi.org/10.1214/16-ECP4748},
      review={\MR{3568347}},
}

\bib{Collins2011}{article}{
      author={\.{Z}yczkowski, Karol},
      author={Penson, Karol~A.},
      author={Nechita, Ion},
      author={Collins, Beno\^{\i}t},
       title={Generating random density matrices},
        date={2011},
        ISSN={0022-2488},
     journal={J. Math. Phys.},
      volume={52},
      number={6},
       pages={062201, 20},
         url={https://doi.org/10.1063/1.3595693},
      review={\MR{2841746}},
}

\end{biblist}
\end{bibdiv}

\end{document}